\documentclass[10pt,a4paper,final]{article}

\usepackage[left=1cm,right=1cm,top=1cm,bottom=1cm,includehead,includefoot]{geometry}
\usepackage{amsmath,amsthm,amssymb,lipsum,setspace,xcolor,graphicx,caption,bm,array,ragged2e,tikz,colortbl,makecell,enumitem}
\usepackage[hidelinks,pdfstartpage=1]{hyperref}
\usepackage[nameinlink]{cleveref}
\usepackage[caption=false]{subfig}

\emergencystretch=\maxdimen
\theoremstyle{plain}
\newtheorem{theorem}{Theorem} [section]
\newtheorem{lemma}{Lemma}[section]
\newtheorem{corollary}{Corollary} [section]
\newtheorem{proposition}{Proposition}[section]

\newtheorem{definition}{Definition}[section]

\title{The Partition Dimension of Corona Product of Complete and Wheel Graph}

\author{M.A.Yusuf, Hasmawati$^{}\footnote{Corresponding author (hasmawati@unhas.ac.id)}$, M.R.Hamidi, A.M.Anwar}

\date{}
\begin{document}
	\maketitle
	\footnotesize {\it Departments of Mathematics,
Hasanuddin University,
Makassar, Indonesia}\\
	\begin{abstract}
		The graph $G$ is a pair of sets $(V(G),E(G))$, where $V(G)$ is a finite set whose elements are called vertices, and $E(G)$ is a set of pairs of members of  $V(G)$, which is called the edges. Let G be a simple graph. For an ordered k-partition $\prod=\{S_1,S_2,…,S_k\}$ of $V(G)$, the representation of $u$ with respect to $\prod$ is k-ordered pairs, $r(u|\prod)=(d(u,S_1 ),d(u,S_2 ),\cdots,d(u,S_k ))$. The partition $\prod$ is called a resolving partition of $G$ if $r(u|\prod)\neq r(v|\prod)$ for all distinct $u,v\in V(G)$. The resolving partition $\prod$ with the minimum cardinality is called minimum resolving partistion. The partition dimension of $G$, denotaded $pd(G)$ is the cardinality of a minimum resolving partition of $G$. In this research, we will determine the partition dimension of corona product of complete graph by using some mathematical statements about resolving partition, consept of equivalent vertices, and same level vertices. There will be several analysis result for the $K_n\bigodot W_m$  vertices referring to equivalent vertices and same level vertices consept. The result show that for $m=n$ $pd(K_n\bigodot W_m )=n ,n\geq3$, for $m=n+1$,  $pd(K_n\bigodot W_m )=3$ ,$n=3$ and  $pd(K_n\bigodot W_m )=n n\geq3$, and for $m=n+2$ ,$pd(K_n\bigodot W_m )=4$,$n=2,3$ and $pd(K_n\bigodot W_m )=n$,$n\geq4$. 
		
\vspace*{5pt}
	\end{abstract}
	
	\section{Introduction}
	The partition dimension was first introduced in 1998 by Chartrand et al., who grouped all the vertices in the graph G into a partition class and determined the representation of each vertex to each partition class [1]. Research related to the theory and determination of the partition dimensions of a graph has been carried out by many researchers. However, the partition dimensions for any graph have not been determined until now. Partition dimensions for certain graphs have been studied by many researchers,for instance, 
\cite{Tomescu84,Tomescu308,Baskoro66,Baskoro06,Yero331,Darmaji}.

	 Let $G =(V(G),E(G))$ be a simple graph. The order of graph G denoted by $|V(G)|$, is the number of vertices on G and size of G denoted by $|E(G)|$, is the number of edges of G. Let $u,v \in V(G)$ be two different vertices in $G$, $uv \in E(G)$ means that $u$ and $v$ are adjacent vertices. The neighborhood $N_G (u)$ of vertex $u$ is the set of vertices adjacent to $u$ in $G$. The degree of u in G denoted by $deg(u)$ is the number of the neighborhood $N_G (u)$. The path on graph $G$ is a series of vertices and edges $v_1,e_1,v_2,e_2,...,v_(n-1),e_(n-1),v_n$  with $e_i  = v_i v_{(i+1)}$,$i = 1,2,...,n-1$. If every two vertices $u$ and $v$ there is always a path containing vertices $u$ and $v$, then the graph $G$ is said to be a connected graph [8]. The length of the path is the number of edges in path. The distance $d(u,v)$ between two vertices $u,v \in V (G)$ is the length of a shortest path between them. For a subset of vertices $S \subset V(G)$, the distance between $u \in V(G)$ denoted by $d(u,S)$ is defined as min$\{d(u,x);x \in s\}$. So, for every $u \in S,d(u,S)=0.$

	 Given an ordered partition $\prod = \{S_1,S_2,...,S_k\}$ of vertices of a graph $G$. The representation of $v \in V(G)$ with respect to $\prod$ denoted by $r(v|\prod )$ is an ordered $k-$tuple $(d(v,S_1),d(v,S_2),...,d(v,S_k))$. An ordered partition $\prod$ is a resolving partition from $V(G)$ if for every $u,v \in V(G)$ has a different representation of $\prod$, that is, $r(u|\prod) \neq  r(v|\prod)$. The resolving partition $\prod$ with the minimum cardinality is called the minimum resolving partition. The partition dimension $pd(G)$ of graph $G$ is the cardinality of the minimum resolving partition of $G$.

	 The complete graph $K_n$ is a simple graph in which every two vertices in $G$ are adjacent, that is for every $u,v \in V(G)$ then $uv \in (G)$. The wheel $W_n$ with $n \geq 3$ is the graph $C_n+K_1$ obtained by adding one center point $x$ and $x$ is adjacent to all vertices in the cycle $C_n$. Thus $W_n$ contains $n+1$ vertices. Let $G$ and $H$ be two connected graphs of order n and m, respectively. The corona product $G \odot H$ is defined as the graphs obtained by copying graph $H$ until the number of graph $H$ is equal to $|V(G)|$, namely $H_1,H_2,...,H_n$ and joining by an edge of every vertex in graph $H_i$ with the $ith-$vertex of $G$. In this paper, we consider the partition dimension of corona product of complete graph $K_n$ and wheel graph $W_m$ where $m=n,m=n+1$, and $m=n+2$ for any integer $n \geq 3$.
	\section{Definitions and Theorms}
	There are several known basic result on the partition dimension of a graph were established. In \cite{Hasmawati-08-22}, Hasmawati, \textit{et al.} established the following definition.
	\begin{definition}
	Let $G$ be a connected graph and $u,v\in{V(G)}$. The vertices $u$ and $v$ are said equivalent, if the vertices statisfy one of the following properties :
\begin{enumerate}
\item $d(u,w)=d(v,w)$ for every $w\in{V(G)}/{\lbrace u,v \rbrace}$,
\item there exist $c\in{V(G)}$ such that $d(u,c)+d(c,s)=d(v,c)+d(c,s)$ for every $s\in{V(G)}/{\lbrace u,v \rbrace}$.
\end{enumerate}
	\end{definition}
	
In Definition 2.1, two vertices $u$ and $v$ that satisfy condition 1 are called strong equivalent vertices and the other are called weak equivalent.
\begin{theorem}
Let $G$ be a connected graph with ordered partition $\Pi$ of vertices of a graph $G$. If $\Pi$ is a resolving partition of $V(G)$ and $u,v\in{V(G)}$  are equivalent in $G$, then $u$ and $v$ or their neighbours respectively belong to distinct class partition  of $\Pi$ . 
\end{theorem}

\begin{definition}
Let $G$ be a connected graph and $u,v\in{V(G)}, k\in{\mathbb{N}}$. The vertices $u$ and $v$ are same level vertices  in $G$ if $d(u,x)=k$ , there exist $y\in{V(G)}$ such that $d(u,x)=d(v,y)=k$ and $|{\lbrace x\in{V(G)}:d(u,x)=k \rbrace}|=|{\lbrace y\in{V(G)}:d(v,y)=k \rbrace}|$. 
\end{definition}

\begin{corollary}
In Definition 2.2, if $x=y$, then the vertices $u$ and $v$ are strong equivalent and $N_G(u)=N_G(v)$.
\end{corollary}

\begin{proof}
Since $x=y$, we have if $d(u,x)=k$, there exist $y=x\in{V(G)}$ such that $d(u,x)=d(v,x)=k$. So, for every $x\in{V(G)}/{\lbrace u,v \rbrace}$,$ d(u,x)=d(v,x)$.  Hence $u$ and $v$ are strong equivalent.

Since $u$ and $v$ are strong equivalent, we will show that $N_G (u)=N_G (v)$ by show that $N_G (u)\subseteq N_G (v)$ and $N_G (v)\subseteq N_G (u)$.
\begin{itemize}
\item Take any $x\in{V(G)}$ and let $x\in{N_G (u)}$. Then $d(u,x)=1$. Since $u$ and $v$ are strong equivalent in $G$, we have $d(v,x)=d(u,x)=1$. Hence $x\in{N_G (v)}$. Therefore  $N_G (u)\subseteq N_G (v)$.
\item Take any $y\in{V(G)}$ and let $y\in{N_G (v)}$. Then $d(v,y)=1$. Since $u$ and $v$ are strong equivalent in $G$, we have $d(u,y)=d(v,y)=1$. Hence $y\in{N_G (u)}$. Therefore  $N_G (v)\subseteq N_G (u)$.
\end{itemize}
Since it has been shown that $N_G(u)\subseteq N_G(v)$ and $N_G(v)\subseteq N_G(u)$, we have $N_G(u)=N_G(v)$.
\end{proof}

For instance, G. Chartrand \textit{et al.} \cite{Chartrand-59-00} established the following theorems.

\begin{theorem}
If $G$ is a graph of order $n\geq 3$ and  diameter $d$, then $g(n,d)\leq pd(G) \leq n-d+1$, with $g(n,d)$ as the least positive integer $k$ for which $(d+1)^{k}\geq n$.
\end{theorem}
	
	\section{Main Result}
	
			In this section, we determine the partition dimension of corona product of complete and wheel graph. In this case we present the partition dimension of graph $K_n\odot W_m$ where $m=n,m=n+1$, and $m=n+2$ for any integer $n \geq 3$, the label of vertex set and the edges set on $K_n\odot W_m$  respectively as follows:
			\[	V(K_n \odot W_m)= \{ V(K_n) \}b\cup \{\bigcup _{i \in V(K_n)} V(W^i_m)\} \]
					and \[E((K_n\odot W_m)=\{E(K_n )\}\cup \{\bigcup_{(i \in V(K_n )} E(W^i_m ))\} \cup \{u_iv_{i,j} |i \leq i < n, 0 \leq j \leq m\}\] where $W^i_m$ is the $i-$th multiplication of the graph $W_m$.
		
		For the need to prove the partition dimension of the graph $K_n\odot W_m$  a set of block vertices  $V( W^{i'}_m)$ is defined where $ V(W^{i'}_m)=V(W^i_m) \cup \{u_i\}$ for $1 \leq i \leq n$. This definition will be used to streamline the writing in the proof elaboration.
			
	\begin{lemma}
		Let $G$ be a connected graph with ordered partition $\Pi$ of vertices of a graph $G$. If $\Pi$ is a resolving partition of  $V(G)$ and $u,v\in{V(G)}$  are the strong equivalent in $G$, then $u$ and $v$ belong to distinct class partition of $\Pi$.
	\end{lemma}
	\begin{proof}
Let $\Pi =\lbrace S_1,S_2,…,S_k \rbrace$ be a resolving partition of $V(G)$. It means that there exist $i\in{\lbrace 1,2,…,k \rbrace},d(u,S_i)\neq d(v,S_i)$ for every $u,v\in{V(G)}$. Let the vertices $u$ and $v$  are the strong equivalent in $G$ and suppose there exist $j\in{\lbrace 1,2,…,k \rbrace}$ such that $u,v\in{S_j}$, then $d(u,S_j )=d(v,S_j )=0$. Next, take any $w\in{V(G)/{\lbrace u,v \rbrace}}$ and let $w\in{S_i}$. Since the vertices $u$ and $v$ are the strong equivalent in $G$, we have $d(u,w)=d(v,w)$. Hence, $d(u,S_i )=d(v,S_i)$. As a result, $r(u|\Pi)=r(v|\Pi)$. So, $\Pi$ is not a resolving partition of $V(G)$, a contradiction. Therefore, if  the vertices $u$ and $v$ are the strong equivalent in $G$, then $u$ and $v$ belong to distinct class partition of $\Pi$. 
\end{proof}

\begin{proposition}
	Given a graph $K_n\odot W_{n+1}$ with $V(K_n\odot W_{n+1} )=\lbrace u_i,v_{i,j} :1\leq i\leq n,0\leq j\leq n+1\rbrace $ and $E(K_n \odot W_{n+1} )=\lbrace u_i u_{i+1},u_i u_{i+2},…,u_i u_n : 1\leq i \leq n-1 \rbrace$ $\cup \lbrace u_i v_{i,j} :1\leq i \leq n,0\leq j \leq n+1\rbrace$ $\cup \lbrace v_{i,0} v_{i,j} : 1\leq i\leq n,1\leq j\leq n+1 \rbrace$ $\cup \lbrace v_{i,j} v_{i,j+1}:1\leq i \leq n+1,1\leq j \leq n+1 \rbrace$ $\cup \lbrace v_{i,1} v_{i,n+1}:1\leq i\leq n \rbrace$, then for every $i\in{\lbrace 1,2,…,n \rbrace }$ the vertices $v_{i,j}$ and $v_{i,k}$ with $j\neq k$,$j,k\in{\lbrace 1,2,…,n+1 \rbrace}$ are the weak equivalent. 
\end{proposition}
\begin{proof}
or any $i\in{\lbrace 1,2,…,n \rbrace }$ and $j,k\in{\lbrace 1,2,…,n+1 \rbrace}$ with $j \neq k$ there exist $v_{i,0}\in {V(K_n \odot W_{n+1})}$ such that for every $ s \in {V(K_n \odot W_{n+1})}/$ $\lbrace v_{i,j},v_{i,k} \rbrace$, $d(v_{i,j},v_{i,0}) +d(v_{i,0},s)$ $=d(v_{i,k},v_{i,0} )$ $+d(v_{i,0},s)$ .
\end{proof}

\begin{corollary}
Given a graph $K_n\odot W_{n+1}$ with $V(K_n\odot W_{n+1} )=\lbrace u_i,v_{i,j} :1\leq i\leq n,0\leq j\leq n+1\rbrace $ and $E(K_n \odot W_{n+1} )=\lbrace u_i u_{i+1},u_i u_{i+2},…,u_i u_n : 1\leq i \leq n-1 \rbrace$ $\cup \lbrace u_i v_{i,j} :1\leq i \leq n,0\leq j \leq n+1\rbrace$ $\cup \lbrace v_{i,0} v_{i,j} : 1\leq i\leq n,1\leq j\leq n+1 \rbrace$ $\cup \lbrace v_{i,j} v_{i,j+1}:1\leq i \leq n+1,1\leq j \leq n+1 \rbrace$ $\cup \lbrace v_{i,1} v_{i,n+1}:1\leq i\leq n \rbrace$,if $\Pi$ is a resolving partition on graph $K_n \odot W_{n+1}$, then the vertices $v_{i,j}$ and $v_{i,k}$ or the vertices $v_{i,j+1},v_{i,k+1},v_{i,j-1},v_{i,k-1}$  belong to distinct class partition of $\Pi$.
\end{corollary}
\begin{proof}
Base on Proposition 3.3, for any $i\in{\lbrace 1,2,…,n \rbrace}$ and any $j,k\in{\lbrace 1,2,…,n+1 \rbrace}$ with $j \neq k$, the vertices $v_{i,j}$ and $v_{i,k}$ are the weak equivalent. Note that the vertices $v_{i,j+1}$ and $v_{i,j-1}$ are the neighbour of $v_{i,j}$ and the vertices $v_{i,k+1}$ and $v_{i,k-1}$ are the neighbour of $v_{i,k}$. Since $\Pi$ is a resolving partition on graph $K_n \odot W_{n+1}$, then by Theorem 2.2, the vertices $v_{i,j}$ and $v_{i,k}$ or the vertices $v_{i,j+1},v_{i,k+1},v_{i,j-1},v_{i,k-1}$  belong to distinct class partition of $\Pi$ . 
\end{proof}

\begin{lemma}
		Let $G$ be a connected graph with ordered partition $\Pi$ of vertices of a graph $G$. Let the vertices $u$ and $v$ are same level vertices on $G$ and there exist $x,y\in {V(G)}$ such that $d(u,x)=d(v,y)=k$. If $\Pi$ is a resolving partition of  $V(G)$ and $k=\min{ \lbrace d(u,a):a\in {S},{S}\in {\Pi} \rbrace},$ then the vertices $u$ and $v$ or $x$ and $y$ belong to distinct class partition  of $\Pi$ .
\end{lemma}
\begin{proof}
Let $\Pi= \lbrace S_1,S_2,…,S_k \rbrace $ be a resolving partition of $V(G)$. It means that there exist $i\in{ \lbrace 1,2,…,k \rbrace },d(u,S_i)\neq d(v,S_i)$ for every $u,v \in {V(G)}$. Let the vertices $u$ and $v$  are in  the same partition class of $\Pi$, said $S_i$, then $d(u,S_i )=d(v,S_i )=0$. Next, take any $j\in{ \lbrace 1,2,…,m \rbrace}$ and choose $x\in {V(G)}/ \lbrace u,v \rbrace $, $x \in {S_j}$ with $d(u,x)=k$. Since $k=\min{ \lbrace d(u,a):a\in {S_j} \rbrace}$ and the vertices $u$ and $v$ are same level vertices on $G$, so there exist $y\in{V(G)}$ such that $d(v,y)=k$. Let $y\in{S_j}$, then $d(u,S_j )=d(v,S_j )=k$. Since $S_j$ is an arbitrary partition class on $\Pi$, we have $d(u,S_j )=d(v,S_j )$ for every $j\in{ \lbrace 1,2,….,m \rbrace }$. As a result, $r(u| \Pi )=r(v| \Pi )$. So, $\Pi$ is not a resolving partition of $V(G)$, a contradiction. Therefore, the vertices $u$ and $v$ or $x$ and $y$ belong to distinct class partition  of $ \Pi $ .
\end{proof}

\begin{proposition}
Given a graph $K_n\odot W_{n+1}$ with $V(K_n\odot W_{n+1} )=\lbrace u_i,v_{i,j} :1\leq i\leq n,0\leq j\leq n+1\rbrace $ and $E(K_n \odot W_{n+1} )=\lbrace u_i u_{i+1},u_i u_{i+2},…,u_i u_n : 1\leq i \leq n-1 \rbrace$ $\cup \lbrace u_i v_{i,j} :1\leq i \leq n,0\leq j \leq n+1\rbrace$ $\cup \lbrace v_{i,0} v_{i,j} : 1\leq i\leq n,1\leq j\leq n+1 \rbrace$ $\cup \lbrace v_{i,j} v_{i,j+1}:1\leq i \leq n+1,1\leq j \leq n+1 \rbrace$ $\cup \lbrace v_{i,1} v_{i,n+1}:1\leq i\leq n \rbrace$,if $\Pi$ is a resolving partition on graph $K_n \odot W_{n+1}$, then :
\begin{enumerate}
\item the vertices $v_{i,0}$ and $v_{k,0}$ for $i\neq k$,   $i,k\in{\lbrace 1,2,…,n \rbrace }$ are same level vertices,
\item for every $j,m\in{\lbrace 1,2,…,n+1 \rbrace}$  the vertices $v_{i,j}$ and $v_{k,m}$ for $i,k \in{\lbrace 1,2,…,n \rbrace }$ are same level vertices,
\item the vertices $u_i$ and $u_j$ for $i\neq j$,$i,j\in{ \lbrace 1,2,…,n \rbrace }$ are same level vertices.
\end{enumerate}
\end{proposition}

\begin{proof}
To prove the proposition we use the Definition 2.2
\begin{enumerate}
\item We will show that the vertices $v_{i,0}$ and $v_{k,0}$ for $i\neq k$,$i,k\in{\lbrace 1,2,…,n \rbrace}$ are same level vertices.
\begin{itemize}
\item For every $x_1 \in{\lbrace u_i,v_{i,j}:1\leq j \leq n+1 \rbrace }$,$d(v_{i,0},x_1 )= 1$, there exist $y_1\in{\lbrace u_k,v_{k,j} :1\leq j}$ $ \leq n+1 \rbrace $ such that $d(v_{i,0},x_1 )=d(v_{k,0},y_1 )=1$ and $|\lbrace u_i,v_{i,j}:1\leq j \leq n+1 \rbrace|=|\lbrace u_k,v_{k,j} :1\leq j\leq n+1 \rbrace | =n+2$,
\item for every $x_2\in{\lbrace u_l:u_l\in{V(K_n \odot W_{n+1}) }/u_i \rbrace }$,$d(v_{i,0},x_2)=2$, there exist $y_2\in{\lbrace u_m :u_m}$ $\in{V(K_n \odot W_{n+1} )/u_k} \rbrace $,  such that $d(v_{i,0},x_2 )=d(v_{k,0},y_2 )=2$ and $|\lbrace u_l : u_l\in{V(K_n \odot}$ $W_{n+1} )/u_i \rbrace|=| \lbrace u_m :u_m \in{V(K_n \odot W_{n+1} )/u_k} \rbrace |=n-1$,
\item for every $x_3\in{\lbrace v_{1,m} :v_{1,m} \in {V(K_n \odot W_{n+1})/v_{i,m}} \rbrace}$ for $m\in{\lbrace 1,2,...,n+1\rbrace},d(v_(i,0),x_3)=3$, there exist $y_3\in{\lbrace v_{c,m}:v_{c,m}\in {V(K_n \odot W_{n+1})/v_{k,m}}\rbrace}$ such that $d(v_{i,0},x_3)=d(v_{k,0},y_3)=3$ and $|\lbrace v_{1,m}:v_{1,m} in {V(K_n \odot W_{n+1})}/$ ${v_{i,m}}\rbrace|=|\lbrace v_{c,m}:v_{c,m}\in{V(K_n \odot W_{n+1})}/v_{k,m}\rbrace|=(n-1)(n-2)$
\end{itemize}
Hence, if $d(v_{i,0},w)=k$, there exist $s\in{V(K_n \odot W_{n+1} )}$ such that $d(v_{k,0},s)=k$ and $|\lbrace w:d(v_{i,0},w)=k \rbrace |= |\lbrace s:d(v_{k,0},s)=k \rbrace |$.
\item We will show that for every $j,m\in{\lbrace 1,2,…,n+1 \rbrace}$  the vertices $v_{i,j}$ and $v_{k,m}$ for $i,k\in{\lbrace 1,2,…,n \rbrace}$ are same level vertices,
\begin{itemize}
\item For every $x_1^{'}\in{\lbrace u_i,v_{i,0},v_{i,j\pm 1} \rbrace }$,$d(v_{i,j},x_1^{'} )=1$, there exist $y_1^{'}\in{\lbrace u_k,v_{k,0},v_{k,m \pm 1} \rbrace }$ such that $d(v_{i,j},x_1^{'} )=d(v_{k,m},y_1^{'} )=1$ and $|\lbrace u_i,v_{i,0},v_{i,j \pm 1} \rbrace|=|\lbrace u_k,v_{k,0},v_{k,m \pm 1} \rbrace |=4$.
\item for every $x_2^{'}\in{\lbrace u_l:u_l\in{V(K_n \bigodot W_{n+1} )/u_i} \rbrace }$ $\cup {\lbrace v_{i,j \pm 2},v_{i,j\pm 3},…,v_{i,n-2} \rbrace}$,$d(v_{i,j},x_2^{'} )=2$, there exist $y_2^{'}\in{\lbrace u_m:}$ $u_m\in{V(K_n \odot W_{n+1} )/u_k} \rbrace $ $\cup {\lbrace v_{k,m\pm 2},v_{k,m\pm 3},…,v_{k,n-2} \rbrace}$ such that $d(v_{i,j},x_2^{'} )=d(v_{k,l},y_2^{'} )=2$ and $|\lbrace u_l :u_l\in{V(K_n \odot W_{n+1} )/u_i} \rbrace \cup {\lbrace v_{i,j \pm 2},v_{i,j\pm 3},…,}$ $v_{i,n-2} \rbrace|=|\lbrace u_m :u_m \in{V(K_n \odot W_{n+1} )/u_k} \rbrace \cup{\lbrace v_{k,m\pm 2},v_{k,m\pm 3},…,v_{k,n-2} \rbrace}|$ $=2n-3$,
\item for every $x_3^{'}\in{\lbrace v_{l,m} :v_{l,m}\in {V(K_n \odot W_{n+1} )/v_{i,j}} \rbrace}$, for $m\in{\lbrace 0,1,2,…,n+1 \rbrace}$,$d(v_{i,j},$ $x_3^{'} )=3$, there exist $y_3^{'}\in {\lbrace v_{c,j} :v_{c,j}\in{V(K_3\odot W_{n+1} )/v_{k,m}} \rbrace}$ for $j\in{\lbrace 0,1,2,…,n+1 \rbrace}$ such that $d(v_{i,j},x_3^{'} )=d(v_{k,m},y_3^{'} )=3$ and $|\lbrace v_{l,m}:v_{l,m}\in {V(K_n \odot W_{n+1} )/v_{i,j}}\rbrace|=|\lbrace v_{c,j}:v_{c,j}\in{V(K_3 \odot W_{n+1} )/v_{k,m}}\rbrace |=(n-1)(n+2)$.
\end{itemize}
Hence, if $d(v_{i,j},w^{'})=k$, there exist $s'\in {V(K_n \bigodot W_{n+1} )}$ such that $d(v_{k,m},s')=k$ and $|\lbrace w^{'}:d(v_{i,j},w^{'})=k \rbrace|=|\lbrace s:d(v_{k,m},s^{'})=k \rbrace |$.
\item We will show that the vertices $u_i$ and $u_j$ for $i\neq j$,$i,j\in{\lbrace 1,2,…,n \rbrace}$ are same level vertices.
\begin{itemize}
\item For every $x_1^{''}\in{\lbrace u_l :u_l\in{V(K_n \odot W_{n+1} )/u_i}\rbrace}$ $\cup {\lbrace v_{i,m}:0\leq m\leq n+1 \rbrace}$,$d(u_i,x_1^{''} )=1$, there exist $y_1^{''}\in{\lbrace u_m :u_m\in}$ ${V(K_n \odot W_{n+1} )/u_j}\rbrace $ $\cup{\lbrace v_{j,m}:0\leq m\leq n+1 \rbrace}$, such that $d(u_i,x_1^{''} )=d(u_j,y_1^{''} )=1$ and $|\lbrace u_l :u_l\in{V(K_n \odot W_{n+1} )}/$ $u_i \rbrace$ $\cup{\lbrace v_{i,m} :0\leq m\leq n+1}\rbrace|$ $=|\lbrace u_m :u_m\in{V(K_n \odot W_{n+1} )/u_j}\rbrace$ $\cup{\lbrace v_{j,m}:0\leq m\leq n+1 \rbrace}|=(n-1)+(n+1)=2n$,
\item for every $x_2^{''}\in{\lbrace v_{l,m} :v_{l,m}\in{V(K_n \odot W_{n+1} )/v_{i,m}}\rbrace }$ for $m\in{\lbrace 0,1,2,…,n+1 \rbrace },d(u_i,$ $x_2^{''} )$ $=2$, there exist $y_2^{''}\in{\lbrace v_{c,m} :v_{c,m}\in{V(K_n \odot W_{n+1} )/v_{j,m}} \rbrace }$ such that $d(u_i,x_2^{''} )=d(u_j,y_2^{''} )$ $=2$ and $|\lbrace v_{l,m} :v_{l,m}\in{V(K_n \odot W_{n+1} )}/$ $v_{i,m} \rbrace|=|\lbrace v_{c,m}:v_{c,m}\in {V(K_n \odot W_{n+1} )}/v_{j,m} \rbrace|$ $=(n-1)(n+2)$. 
\end{itemize}
Hence, if $d(u_i,w^{''})=k$, there exist $s^{''}\in {V(K_n \odot W_{n+1} )}$ such that $d(u_j,s^{''})=k$ and $|\lbrace w^{''}:d(u_i,w^{''})=k \rbrace|=|\lbrace s^{''}:d(u_j,s^{''})=k \rbrace |$.
\end{enumerate}
\end{proof}

\begin{lemma}
	For every $m \geq n \geq 3, pd(K_n \odot W_m) \geq n$
\end{lemma}
\begin{proof}
Let $\Pi={\lbrace S_1,S_2,\ldots,S_{n-1} \rbrace}$ be a ordered partition of $V\left(K_n\odot W_{m}\right)$ and for every $x,y\in V\left(K_n\odot W_{m}\right)$  partitioned arbitrarily to $S_1,S_2,\ldots,S_{n-1}$. Based on Proposition 3.2 the vertices $u_i$ and $u_j$ for $i\neq j$,$i,j\in{\lbrace 1,2,\ldots,n \rbrace}$ are same level vertices. Since there are only $n-1$ partition classes on $\Pi$ then there are at least two level vertices, namely $u_i$ and $u_j$ which are in the same partition class. Not that for every $w_i\in{\lbrace u_l:u_l\in V\left(K_n\odot W_{m}\right)/u_i \rbrace}$ $\cup\left\{v_{i,j}:0\le j\le m\right\},d\left(u_i,w_i\right)=1$, there exist $w_j\in{u_t\:u_t\in V\left(K_n\odot W_{m}\right)/u_j}\cup\left\{v_{j,m}:0\le m\le n+1\right\}$, such that $d\left(u_i,w_i\right)=d\left(u_j,w_j\right)=1$. Consider the following cases.

\textbf{Case 1} If for any $j\in\left\{1,2,\ldots,n-1\right\},\ \ S_j$ contains $w_i$ and $w_j$, then $d\left(u_i,S_j\right)=d\left(u_j,S_j\right)=1$. So, $r\left(u_i\middle|\Pi\right)=r\left(u_j\middle|\Pi \right)$. Therefore, $\Pi$ is esolving partition of graph $K_n\odot W_{m}$.

\textbf{Case 2} If there exist $j\in\left\{1,2,\ldots,n-1\right\},\ \ S_j$ does not contain $w_i$ and $w_j$, then for any $k\in\left\{1,2,\ldots,m\right\}$ we can choose the vertex $w_i^\prime\in V\left(K_n\odot W_{m}\right),\ w_i^\prime\in S_k$, such that $d\left(u_i,w_i^\prime\right)=l$, where $l=\min{\left\{d\left(u_i,a\right)\:a\in S_k\right\}}$. Since the vertices $u_i$ and $u_j$ are same level vertices on $K_n\odot W_{m}$, then there exist $w_j^\prime\in V\left(K_n\odot W_{m}\right),\ w_j^\prime\in S_k$, such that $d\left(u_j,w_j^\prime\right)=l$. Based on Lemma 3.1, we have $\Pi$ is not a resolving partition of graph $K_n\odot W_{m}$.

So, because there are only $n-1$ partition classes, then each partition class on $\Pi$ contains a pair of same level vertices or contains the vertices that have the same distance from pairs of level vertices. Therefore, $\Pi$ is not a resolving partition of graph $K_n\odot W_{n+1}$, we have $pd\left(K_n\odot W_{m}\right)\neq n-1$. Consequently, the lower bound of the partition dimension of the graph $K_n\odot W_{n+1}$ is $pd\left(K_n\odot W_{m}\right)\geq n$.
\end{proof}

\begin{lemma}
	For Every $n \geq 3, pd(K_n \odot W_n) \leq n$
\end{lemma}
\begin{proof}
	Based on several propositions and lemmas regarding same-level points, equivalent points, and their relationships with the partition dimension of the graph $(K_n \odot W_n)$  we will construct partition  $\prod =\{S_1, S_2, \cdots S_n\}$. For each vertex in $ V(W^{i'}_m)$ it will be partitioned into all partitions in $\prod$ except the partition $S_i$, we defined the partitions below \\
	$ S_1=\{u_i, v_{i,0}, v_{i,1}, v_{i,2}\}$, where $2 \leq i \leq n$ \\
	$ S_2=\{u_1, v_{1,0}, v_{1,1}, v_{1,2}, v_{1,3}\}$, where $3 \leq i \leq n$ \\
		$ S_x=\{v_{i,x}, v_{i,x+1}\}$, where $3 \leq x \leq n-1$, $1 \leq i \leq x-1$, $x+1 \leq j \leq n$  \\
		$S_{n-1}=\{v_{i,n-2}, v_{n,n}| 1 \leq i \leq n-2\}$\\
			$S_{n1}=\{v_{i,n}| 1 \leq i \leq n-1\}$\\

\noindent We will $K_n \odot W_n$ prove that $\prod$ is a resolving partition of graph . To prove this, we will show for consider two distinct vertices $x, y \in V(K_n \odot W_n)$ have a different representation respect to $\prod$. It will be shown that for any  $i,j \in \{1,2,\cdots n\}, i \neq j$ the vertices in $V(W^{i'}_m)$ and $V(W^{j'}_m)$ have different representations with respect to . Notice that because there is no vertex in  $V(W^{j'}_m)$ partition into $S_i$ for any  $x \in V(W^{j'}_n)$, $r(x, \prod )= k$ where $k=\{0,1\}$. Whereas for the vertices in $V(W^{j'}_n)$, there is at least one vertex partitioned into $S_i$. Therefore, for any $y \in V(W^{j'}_n)$, $r(y, \prod )= k$ where $k=\{0,1\}$. Thus $r(x|\prod ) \neq r(x|\prod )$. Hence, it is proven that for any  $i,j \in \{1,2,\cdots n\} i \neq j$ the vertices in $V(W^{i'}_n)$ and $V(W^{j'}_n)$ have different representation. Then, it remains to show that vertices partitioned into the same partition class have different representations. It is unnecessary to show vertices partitioned into different partition classes because if any vertex $x_i \in S$ is taken, only the $i$-th entry in $r(x|\prod)$ will be zero, while the others are non-zero. \\
\noindent For $V(W^{1'}_n)$, the representative of vertices that have the same partition class respect to  $\prod $ are $r(v_{1,0}|\prod)=\{2,0,  \underbrace{1}_\textit{3rd}, \cdots,\underbrace{1}_\textit{n-th}\}$, $r(v_{1,1}|\prod)=\{2,0,  \underbrace{2}_\textit{3rd}, \cdots,\underbrace{1}_\textit{n-th}\}$, 
$r(v_{1,2}|\prod)=\{2,0,  \underbrace{1}_\textit{3rd}, \cdots,\underbrace{2}_\textit{n-th}\}$, and $r(u_1|\prod)=\{1,0,  \underbrace{1}_\textit{3rd}, \cdots,\underbrace{1}_\textit{n-th}\}$. Thus, it is proven that all vertices in $V(W^{1'}_n)$ have different representations in  $\prod$.\\

 For $V(W^{2'}_n)$, the representative of vertices that have the same partition class respect to  $\prod $  are $r(v_{2,0}|\prod)=\{0,  \underbrace{1}_\textit{2nd},1 \cdots,\underbrace{1}_\textit{n-th}\}$, $r(v_{2,1}|\prod)=\{0,  \underbrace{2}_\textit{2nd},2 \cdots,\underbrace{1}_\textit{n-th}\}$, 
$r(v_{2,2}|\prod)=\{0,  \underbrace{2}_\textit{2nd}, \cdots,\underbrace{2}_\textit{n-th}\}$, and $r(u_2|\prod)=\{0,  \underbrace{1}_\textit{2nd}, \cdots,\underbrace{1}_\textit{n-th}\}$. Thus, it is proven that all vertices in $V(W^{2'}_n)$ have different representations in  $\prod$.\\

 For $V(W^{i'}_n)$, $3 \leq i \leq n$ the representative of vertices that have the same partition class respect to  $\prod $  are $r(v_{i,0}|\prod)=\{0, 1 \cdots, \underbrace{3}_\textit{i-th},1 \cdots,\underbrace{1}_\textit{n-th}\}$, $r(v_{i,1}|\prod)=\{0, 2,\cdots,  \underbrace{3}_\textit{i-th},1 \cdots,\underbrace{1}_\textit{n-th}\}$,
 
$r(v_{i,2}|\prod)=\{0, 1 \cdots  \underbrace{3}_\textit{i-th},1 \cdots,\underbrace{2}_\textit{n-th}\}$, and $r(u_{i}|\prod)=\{0, 1 \cdots  \underbrace{2}_\textit{i-th},1 \cdots,\underbrace{1}_\textit{n-th}\}$,. Thus, it is proven that all vertices in $V(W^{i'}_n)$, $3 \leq i \leq n$ have different representations in  $\prod$.\\

\noindent We can see that for every two vertices have a different representation respect to  $\prod$, then $\prod$ is resolving partition to $K_n \odot W_n$.  Consequently $pd(K_n \odot W_n) \leq n$.
\end{proof}
\begin{lemma}
	For Every $n \geq 4, pd(K_n \odot W_{n+1}) \leq n$
\end{lemma}
\begin{proof}
Will define an ordered partition $\Pi^\prime=\left\{S_1^\prime,S_2^\prime,\ldots,S_n^\prime\right\}$ of $V\left(K_n\odot W_{n+1}\right)$. Based on Proposition 3.1, for every $i\in\left\{1,2,..,n\right\}$, the vertices $v_{i,j}$ and $v_{i,k}$ where $j\neq k,\ \ j,k\in\left\{1,2,\ldots,n+1\right\}$ are the weak equivalent vertices, so by Corollary 3.1, the vertices $v_{i,j}$ and $v_{i,k}$ or the vertices $v_{i,j+1},v_{i,k+1},v_{i,j-1},v_{i,k-1}\ $ belong to distinct class partition of $\Pi^\prime$. According Proposition 3.2, for every $m\in\left\{1,2,\ldots,n+1\right\}\ $ the vertices $v_{i,m}$ and $v_{j,m}$ for $i,j\in\left\{1,2,\ldots,n\right\}$ are same level vertices and there exist $x_1,y_1\in V\left(K_n\odot W_{n+1}\right)$, such that $d\left(v_{i,m},x_1\right)=d\left(v_{j,m},y_1\right)=k_1$, where $k_1=\min{\lbrace d(v_{i,m},a):a\in{S}, S\in{\Pi^{'}}\rbrace}$. So, by Lemma 3.2, the vertices $v_{i,m}$ and $v_{j,m}$ or the vertices $x_1$ and $y_1$ belong to distinct class partition  of  $\Pi^\prime$. Using the same proposition and lemma, we have the vertices $v_{i,0}$ and $v_{j,0}$ are same level vertices and there exist $x_2,y_2\in V\left(K_n\odot W_{n+1}\right)$, such that $d\left(v_{i,0},x_2\right)=d\left(v_{j,0},y_2\right)=k_2$, where $k_2=\min{\lbrace d(v_{i,0},a):a\in{S}, S\in{\Pi^\prime}\rbrace}$, then the vertices $v_{i,0}$ and $v_{j,0}$ or the vertices $x_2$ and $y_2$ belong to distinct class partition  of  $\Pi^\prime$. Using this information, we define :

$S_1^\prime=\left\{u_i,v_{i,0},v_{i,1},v_{i,2}\right\}$, where $2\le i\le n$ ,

$S_2^\prime=\left\{u_1,v_{1,0},\ v_{1,1},v_{1,2},v_{i,3}\right\}$, where $3\le i\le n$ ,

$S_x^\prime=\lbrace v_{i,x},v_{j,x+1}\rbrace$,where $\ \ \ 3\le x\le n-1,\ \ 1\le i\le x-1,x\ +1\le j\le n$ ,

$S_{n-1}^\prime=\lbrace v_{i,n-1} \rbrace $ $\cup{\lbrace v_{n,n},v_{n,n+1} \rbrace}$, and 

$S_n^\prime=\lbrace v_{i,n},v_{i,n+1} |1\leq i \leq n \rbrace$.

We will prove that $\Pi^\prime$ is a resolving partition of graph $K_n\odot W_{n+1}$. To prove this, consider two distinct vertices $x,y\in V\left(K_n\odot W_{n+1}\right)$ that are contained in the same partition class on $\Pi^\prime$.
 
For the vertex $u_2\in S_1^\prime,r(u_2 |\Pi^\prime)=\lbrace 0,1,{\underbrace{1,\ldots,1}_{n-2}} \rbrace$ and for the vertex $v_{2,0}\in S_1^\prime,\ r\left(v_{2,0}\right|\Pi^\prime)= \lbrace 0,2,{\underbrace{1,\ldots,1}_{n-2}} \rbrace$. For any $u_i,v_{i,0}\in S_1^\prime/\left\{u_2,v_{2,0}\right\}$, representation the vertices $u_i$ and $v_{i,0}$ with respect to $\Pi^\prime$ are each distinguished by the partition class $S_i^\prime$ for $i\in\left\{3,4,\ldots,n\right\}$. In this case :
$r\left(u_3\right|\Pi^\prime)=\lbrace 0,1,2,{\underbrace{1,\ldots,1}_{n-3}} \rbrace$,
$r\left(u_4\right|\Pi^\prime)=\lbrace 0,1,1,2,{\underbrace{1,\ldots,1}_{n-4}} \rbrace$,
$r\left(u_i\right|\Pi^\prime)=\lbrace 0,{\underbrace{1,\ldots,1}_{i-2}},2,$ ${\underbrace{1,\ldots,1}_{n-i}}  \rbrace$,  for  $5\le i\le n$;
$r\left(v_{3,0}\right|\Pi^\prime)=\lbrace 0,1,3,{\underbrace{1,\ldots,1}_{n-3}} \rbrace$,
$r\left(v_{4,0}\right|\Pi^\prime)=\lbrace 0,1,1,3,{\underbrace{1,\ldots,1}_{n-4}}\rbrace$, and
$r\left(v_{i,0}\right|\Pi^\prime)=\lbrace 0,{\underbrace{1,\ldots,1}_{i-2}},3,{\underbrace{1,\ldots,1}_{n-i}} \rbrace$,  for  $5\le i\le n$.
So, for any $x_1,y_2\in S_1^\prime,r\left(x_1\right|\Pi^\prime)\neq r\left(y_2\right|\Pi^\prime)$. 

For the vertices $u_1,v_{1,0}\in S_2^\prime$, representation of the vertices with respect to $\Pi^\prime$ is distinguished by the partition class $S_1^\prime$.  In this case $d\left(u_1,S_1^\prime\right)=1$, while $d\left(v_{1,0},S_1^\prime\right)=2$. So, $d\left(u_1,S_1^\prime\right)\neq d\left(v_{1,0},S_1^\prime\right)$. For the vertices $v_{1,1},v_{1,2}\in S_2^\prime$, representation of the vertices with respect to $\Pi^\prime$ is distinguished by the partition class $S_n^\prime$. We get $d\left(v_{1,1},S_n^\prime\right)=1$ and $d\left(v_{1,2},S_n^\prime\right)=2$. For $v_{i,3}\in S_2^\prime $ where $3\le i\le n$, representation of the vertices with respect to $\Pi^\prime$ are :
$r\left(v_{3,3}\right|\Pi^\prime)=\lbrace 1,0,3,1,{\underbrace{2,\ldots,2}_{n-4}} \rbrace$, and
$r\left(v_{i,3}\right|\Pi^\prime)=\lbrace 1,0,1,{\underbrace{2,\ldots,2}_{i-4}},3,{\underbrace{2,\ldots,2}_{n-i}}\rbrace$,  for $4\le i\le n$.
So, for any $x_2,y_2\in S_2^\prime,r\left(x_2\right|\Pi^\prime)\neq r\left(y_2\right|\Pi^\prime)$. 

For $v_{i,x}\in S_x^\prime $, where $3\le x\le n-2,\ 1\le i\le x-1$ representation of the vertices $v_{i,x}\in S_x^\prime $ with respect to $\Pi^\prime$ are :
$r\left(v_{1,x}\right|\Pi^\prime)= \lbrace 2,{\underbrace{1,\ldots,1}_{x-2}},0,1,{\underbrace{2,\ldots,2}_{n-x-1}}\rbrace $, for $x=3,4$;
$r\left(v_{1,x}\right|\Pi^\prime)=\lbrace 2,{1,\underbrace{2,\ldots,2}_{x-4}},1,0,1,{\underbrace{2,\ldots,2}_{n-x-1}} \rbrace$, for $5\le x\le n-2$;
$r\left(v_{2,x}\right|\Pi^\prime)=\lbrace 1,2,{\underbrace{1,\ldots,1}_{x-3}},0,1,{\underbrace{2,\ldots,2}_{n-x-1}} \rbrace$, for $x=3,4$;
$r\left(v_{2,x}\right|\Pi^\prime)=\lbrace {1,\underbrace{2,\ldots,2}_{x-3}},1,0,1,{\underbrace{2,\ldots,2}_{n-x-1}}\rbrace $, for $5\le x\le n-2$;
$r\left(v_{i,x}\right|\Pi^\prime)=\lbrace 1,{\underbrace{2,\ldots,2,}_{i-2}},3,$ ${\underbrace{2,\ldots,2,}_{x-i-2}},1,0,1,{\underbrace{2,\ldots,2,}_{n-x-1}}\rbrace $, 
for $5\le x\le n-2, 3\le i\le x-2$, and
$r\left(v_{x-1,x}\right|\Pi^\prime)= \lbrace 1,{\underbrace{2,\ldots,2,}_{x-4}}1,$ $3,0,1,{\underbrace{2,\ldots,2}_{n-x-1}} \rbrace$ for $4\le x\le n-2$.

As for the vertices $v_{j,x+1}\in S_x^\prime $, where $3\le x\le n-2,\ x+1\le j\le n$ representation of the vertices $v_{j,x+1}\in S_x^\prime $ with respect to $\Pi^\prime$ are :
$r\left(v_{j,x+1}\right|\Pi^\prime)=\lbrace 1,{\underbrace{2,\ldots,2}_{x-3}},1,0,3,1,{\underbrace{2,\ldots,2}_{n-x-2}} \rbrace$, for $j=x+1$;
$r\left(v_{j,x+1}\right|\Pi^\prime)=\lbrace 1,{\underbrace{2,\ldots,2}_{x-3}},1,0,1,3,{\underbrace{2,\ldots,2}_{n-x-2}} \rbrace$, for $j=x+2$, and
$r\left(v_{j,x+1}\right|\Pi^\prime)=\lbrace 1,{\underbrace{2,\ldots,2}_{x-3}},1,0,1,{\underbrace{2,\ldots,2}_{j-x-2}},3,{\underbrace{2,\ldots,2}_{n-j}} \rbrace$, for $x+3\le j\le n$.	
So for any $x_3,y_3\in S_x^\prime,r\left(x_3\right|\Pi^\prime)\neq r\left(y_3\right|\Pi^\prime)$.

For the vertices $v_{i,n-1},v_{n,n},v_{n,n+1}\in S_{n-1}^\prime$, where $n\geq4,\ 1\le i\le n-2$ representation of the vertices with respect to $\Pi^\prime$ are :
$r\left(v_{1,n-1}\right|\Pi^\prime)=\lbrace 2,{\underbrace{1,\ldots,1}_{n-3}},0,1 \rbrace$, for $n=5,6$;
$r\left(v_{1,n-1}\right|\Pi^\prime)=\lbrace 2,{1,\underbrace{2,\ldots,2}_{n-5}},1,0,1 \rbrace $, for $n\geq7$;
$r\left(v_{2,n-1}\right|\Pi^\prime)=\lbrace 1,2,{\underbrace{1,\ldots,1}_{n-3}},0,1 \rbrace$, for $n=5,6$;
$r\left(v_{2,n-1}\right|\Pi^\prime)=\lbrace 1,{\underbrace{2,\ldots,2}_{n-4}},1,0,1 \rbrace$, for $n\geq7$;
$r\left(v_{i,n-1}\right|\Pi^\prime)=\lbrace 1,{\underbrace{2,\ldots,2}_{i-2}},3,{\underbrace{2,\ldots,2}_{n-i-3}},1,$ $0,1 \rbrace$, for $n\geq4,\ 3\le i\le n-2$;
$r\left(v_{n-2,n-1}\right|\Pi^\prime)= \lbrace 1,{\underbrace{2,\ldots,2,}_{n-5}}1,3,0,1\rbrace$, 
$r\left(v_{n,n}\right|\Pi^\prime)=\lbrace 1,{\underbrace{2,\ldots,2,}_{n-4}}1$ $,0,3\rbrace $, and 
$r\left(v_{n,n+1}\right|\Pi^\prime)=\lbrace 1,{\underbrace{2,\ldots,2,}_{n-3}}0,3 \rbrace$.
So for any $x_4,y_4\in S_{n-1}^\prime$, $r\left(x_4|\prod\prime\right)$ $\neq r\left(y_4\right|\Pi^\prime)$. 

For $v_{i,n},v_{i,n+1}\in S_n^\prime $, where $n\geq4,\ 1\le i\le n-1$ representation of the vertices with respect to $\Pi^\prime$ are :
$r\left(v_{1,n}\right|\Pi^\prime)=\lbrace 2,1,{\underbrace{2,\ldots,2}_{n-4}},1,0 \rbrace $,
$r\left(v_{2,n}\right|\Pi^\prime)=\lbrace 1,{\underbrace{2,\ldots,2}_{n-3}},1,0\rbrace$, 
$r\left(v_{i,n}\right|\Pi^\prime)=\lbrace 1,{\underbrace{2,\ldots,2}_{i-2}},3,{\underbrace{2,\ldots,2}_{n-i-2}},1,0 \rbrace $, for $n\geq4,\ \ 3\le i\le n-2$;
$r\left(v_{i-1,n}\right|\Pi^\prime)=\lbrace 1,{\underbrace{2,\ldots,2}_{n-4}},1,3,0\rbrace $, 
$r\left(v_{1,n+1}\right|\Pi^\prime)=\lbrace 2,1,{\underbrace{2,\ldots,2}_{n-3}},0\rbrace $, 
$r\left(v_{2,n+1}\right|\Pi^\prime)=\lbrace 1,{\underbrace{2,\ldots,2}_{n-2}},0\rbrace $,and
$r\left(v_{i,n+1}\right|\Pi^\prime)=\lbrace 1,$ ${\underbrace{2,\ldots,2}_{i-2},}$ $3,{\underbrace{2,\ldots,2}_{n-i-1}},0\rbrace$, for $3\le i\le n-1$.
So for any $x_5,y_5\in S_n^\prime $,$r\left(x_5\right|\Pi^\prime)\neq r\left(y_5\right|\Pi^\prime)$.

We can see that for every two vertices $x,y\in V\left(K_n\odot W_{n+1}\right)$, we have $r(x|\Pi^\prime)\neq r(y|\Pi^\prime)$. Therefore $\Pi^\prime$ is a resolving partition of graph $K_n\odot W_{n+1}$. Consequently, the upper bound of the partition dimension of the graph $K_n\odot W_{n+1}$ is $pd\left(K_n\odot W_{n+1}\right)\le n$. Based on the lower bound and upper bound of of the graph $K_n\odot W_{n+1}$ for $n\geq 4$, we get $pd\left(K_n\odot W_{n+1}\right)=n$. 
\end{proof}

\begin{lemma}
	For Every $n \geq 4, pd(K_n \odot W_{n+2}) \leq n$
\end{lemma}
\begin{proof}
	With the same consideration as in Lemma 3.4,  we defined the partitions below \\
	$ S_1=\{u_i, v_{i,0}, v_{i,1}, v_{i,2},v_{i,n+2}\}$, where $2 \leq i \leq n$ \\
	$ S_2=\{u_1, v_{1,0}, v_{1,1}, v_{1,2}, v_{1,3},v_{i,n+2} \}$, where $3 \leq i \leq n$ \\
	$ S_x=\{v_{i,x}, v_{i,x+1}\}$, where $3 \leq x \leq n-1$, $1 \leq i \leq x-1$, $x+1 \leq j \leq n$  \\
	$S_{n-1}=\{v_{i,n-2}| 1 \leq i \leq n-2\} \cup \{v_{n,n}, v_{n, n+1}\}$\\
	$S_{n1}=\{v_{i,n}, v_{i,n+1}| 1 \leq i \leq n-1\}$\\
	The same concept applied in Lemma 3.4 will be used to show that points within the same block and partitioned into a certain partition class have different representations.
	\noindent For $V(W^{1'}_n)$, the representative of vertices that have the same partition class respect to  $\prod $ are $r(v_{1,0}|\prod)=\{2,0,  \underbrace{1}_\textit{3rd}, \cdots,\underbrace{1}_\textit{n-th}\}$, $r(v_{1,1}|\prod)=\{2,0,  \underbrace{2}_\textit{3rd}, \cdots,\underbrace{2}_\textit{n-th}\}$, 
 $r(v_{1,2}|\prod)=\{2,0,  \underbrace{1}_\textit{3rd}, \cdots,\underbrace{2}_\textit{n-th}\}$, 	$r(v_{1,n}|\prod)=\{2,\underbrace{2}_\textit{2nd}, \cdots, \underbrace{1}_\textit{(n-1)-th}, \cdots,\underbrace{0}_\textit{n-th}\}$, 	$r(v_{1,n+1}|\prod)=\{1,\underbrace{1}_\textit{2nd}, \cdots, \underbrace{2}_\textit{(n-1)-th}, \cdots,\underbrace{0}_\textit{n-th}\}$,
	 and $r(u_1|\prod)=\{1,0,  \underbrace{1}_\textit{3rd}, \cdots,\underbrace{1}_\textit{n-th}\}$. Thus, it is proven that all vertices in $V(W^{1'}_n)$ have different representations in  $\prod$.\\
	 
 For $V(W^{2'}_n)$, the representative of vertices that have the same partition class respect to  $\prod $  are $r(v_{2,0}|\prod)=\{0,  \underbrace{2}_\textit{2nd},1 \cdots,\underbrace{1}_\textit{n-th}\}$, $r(v_{2,1}|\prod)=\{0,  \underbrace{2}_\textit{2nd},2 \cdots,\underbrace{1}_\textit{n-th}\}$, 
	 $r(v_{2,2}|\prod)=\{0,  \underbrace{2}_\textit{2nd}, \cdots,\underbrace{2}_\textit{n-th}\}$, $r(v_{2,n}|\prod)=\{1,\underbrace{2}_\textit{2nd}, \cdots, \underbrace{1}_\textit{(n-1)-th}, \cdots,\underbrace{0}_\textit{n-th}\}$, 	$r(v_{2,n+1}|\prod)=\{1,\underbrace{2}_\textit{2nd}, \cdots, \underbrace{2}_\textit{(n-1)-th}, \cdots,\underbrace{0}_\textit{n-th}\}$,
	 and $r(u_2|\prod)=\{0,  \underbrace{1}_\textit{3rd}, \cdots,\underbrace{1}_\textit{n-th}\}$.
	  Thus, it is proven that all vertices in $V(W^{2'}_n)$ have different representations in  $\prod$.\\
	  
 For $V(W^{i'}_n)$, $3 \leq i \leq n$ the representative of vertices that have the same partition class respect to  $\prod $  are $r(v_{i,0}|\prod)=\{0, 1 \cdots, \underbrace{3}_\textit{i-th},1 \cdots,\underbrace{1}_\textit{n-th}\}$, $r(v_{i,1}|\prod)=\{0, 2,\cdots,  \underbrace{3}_\textit{i-th},1 \cdots,\underbrace{1}_\textit{n-th}\}$,
	  
	  $r(v_{i,2}|\prod)=\{0, 1 \cdots  \underbrace{3}_\textit{i-th},1 \cdots,\underbrace{2}_\textit{n-th}\}$,
	  $r(v_{i,n}|\prod)=\{1,1  \cdots, \underbrace{1}_\textit{(n-1)-th}, \cdots,\underbrace{0}_\textit{n-th}\}$, 	$r(v_{2,n+1}|\prod)=\{1,1, \cdots, \underbrace{2}_\textit{(n-1)-th}, \cdots,\underbrace{0}_\textit{n-th}\}$,
	   and $r(u_{i}|\prod)=\{0, 1 \cdots  \underbrace{2}_\textit{i-th},1 \cdots,\underbrace{1}_\textit{n-th}\}$,. Thus, it is proven that all vertices in $V(W^{i'}_n)$, $3 \leq i \leq n$ have different representations in  $\prod$.\\
	   \noindent We can see that for every two vertices have a different representation respect to  $\prod$, then $\prod$ is resolving partition to $K_n \odot W_{n+2}$.  Consequently $pd(K_n \odot W_{n+2}) \leq n$.
\end{proof}
\begin{theorem}
For Every $n, n \geq 3, pd(K_n \odot W_{n}) = n$
\end{theorem}
\begin{proof}
	Based on Lemma 3.3, it is obtained that $pd(K_n \odot W_{n}) > n-1$.And based on Lemma 3.4, it is obtained that $pd(K_n \odot W_{n}) \leq  n$.. thus it can be concluded that $pd(K_n \odot W_{n}) \leq  n$ for $n \geq 3$.
\end{proof}
\begin{theorem}
	For Every $n, n \geq 3$
	\begin{align*}
		pd(K_n \odot W_{n+1}) =\begin{cases}
			4, n=3 \\
			n, n \geq 4
		\end{cases}
	\end{align*} 
\end{theorem}
\begin{proof}
	For $n=3$\\
	Let graph $K_3 \odot W_4$ with $V(K_3 \odot W_4)=\lbrace u_i,v_{i,j}:1\leq i \leq 3,0\leq j\leq 4\rbrace $ and $E(K_3 \odot W_4)=\lbrace u_1 u_2,u_1 u_3,u_2 u_3\rbrace$ $\cup \lbrace u_i v_{i,j}:1\leq i\leq 3,0\leq j\leq 4 \rbrace$ $\cup \lbrace v_{i,0} v_{i,j} : 1\leq i\leq 3,1\leq j\leq 4\rbrace$  $\cup \lbrace v_{i,j} v_{i,j+1}:1\leq i\leq 3,1\leq j\leq 4 \rbrace$ $\cup \lbrace v_{i,1} v_{i,n+1}:1\leq i\leq 3 \rbrace $.

We will show that $pd(K_3 \odot W_4 )=4$. Let $\Pi= \lbrace S_1,S_2,S_3 \rbrace$ be an ordered partition of $V(K_3 \odot W_4 )$ and for every $x,y\in{V(K_3 \odot W_4 )}$ partitioned arbitrarily to $S_1,S_2$, and $S_3$. Based on Proposition 3.2, for every $i\in{\lbrace 1,2,3 \rbrace }$, the vertices $v_{i,1}$ and $v_{i,3}$ are the strong equivalent. Using the same proposition, we have the vertices $v_{i,2}$ and $v_{i,4}$ are also the strong equivalent and by Corollary 2.1, $N(v_{i,1} )=N(v_{i,3} )$ and $N(v_{i,2} )=N(v_{i,4} )$. Consider the following cases.

\textbf{Case 1} If there is a partition class on $\Pi$ that contais two strong equivalent vertices, namely $v_{i,1},v_{i,3}\in {S_i }$ or $v_{i,2},v_{i,4}\in {S_i}$, then there are two strong equivalent vertices which is in the same partition class on $\Pi$. Consequently, based on Proposition 3.1, we have $\Pi $ is not a resolving partition of graph $K_3 \odot W_4$.

\textbf{Case 2} If for any $i\in{\lbrace 1,2,3 \rbrace}$, the vertices $v_{i,a},v_{i,b}\in {S_i}$ with $a,b\in{\lbrace 0,1,2,3,4 \rbrace}$ and $v_{i,a},v_{i,b}$ are not strong equivalent vertices, then :

\begin{itemize}
\item if for any $j\in{\lbrace 1,2,3 \rbrace}$,$S_j$ contains the neighbors of $v_{i,a}$ and $v_{i,b}$ then $d(v_{i,a},S_j )=d(v_{i,b},S_j )$ $=1$,  so $r(v_{i,a}|\Pi)=r(v_{i,b}|\Pi)$. Therefore, $\Pi$ is not a resolving partition of graph $K_3 \odot W_4$.
\item if there exist $j\in{\lbrace 1,2,3 \rbrace}$,$S_j$ does not contain the neighbors of $v_{i,a}$ and $v_{i,b}$, then there are at least two level vertices, namely $v_{i,a}$ and $v_{j,b}$ which are in the same partition class on $\Pi$. Said $v_{i,a},v_{j,b}\in{S_j}$. So that, for any $l\in{\lbrace 1,2,3 \rbrace}$ we can choose $c\in{V(K_3 \odot W_4 )},c\in{S_l}$, such that $d(v_{i,a},c)=k$, where $k=\min {\lbrace d(v_{i,a},a):a\in{S_l} \rbrace}$. Since the vertices $v_{i,a}$ and $v_{j,b}$ are same level vertices on $K_3 \odot W_4$, then there exist $c^{'}\in{V(K_3 \odot W_4 )}$,$c^{'}\in{S_l}$, such that $d(v_{j,b},c^{'})=k$. Based on Lemma 3.1, we have $\Pi$ is not a resolving partition of graph $K_3 \odot W_4$.
\end{itemize}

So, because there are only three partition classes, then each partition class on $\Pi$ contains a pair of equivalent vertices or each neighbor of the pair of equivalent vertices is in the same partition class. In this case, there are at least two vertices, namely $v_{i,a}$ and $v_{i,b}$ where $a,b\in{\lbrace 0,1,2,3,4 \rbrace}$ are in the same partition class. Therefore, $\Pi$ is not a resolving partition of graph $K_3 \odot W_4$, we have $pd(K_3 \odot W_4 ) \neq 3$. Consequently, the lower bound of the partition dimension of the graph $K_3 \odot W_4$ is $pd(K_3 \odot W_4 )\geq4 $.

Next define an ordered partition $\Pi^\prime= \lbrace S_1^\prime,S_2^\prime,S_3^\prime,S_4^\prime \rbrace $ of $V(K_3 \odot W_4 )$. Defined :
$S_1^\prime= \lbrace u_2,u_3,v_{2,0},v_{3,0},v_{2,1},v_{2,2},v_{3,1},v_{3,2 } \rbrace$, $S_2^\prime=\lbrace u_1,v_{1,0},v_{1,1},v_{1,2},v_{3,3} \rbrace $, $S_3^\prime= \lbrace v_1,3,v_{2,3} \rbrace $,$S_4^\prime= $ $\lbrace v_{1,4},$ $v_{2,4},v_{3,4} \rbrace $. Now, we wiil investigate the representation of the vertices on graph $K_3 \odot W_4$ with respect to $\Pi^\prime$. We have $r(u_1|\Pi^\prime)= \lbrace 1,0,1,1 \rbrace$; $r(u_i|\Pi^\prime)= \lbrace 0,1,i-1,1 \rbrace$,  for  $i=2,3$; $r(v_{1,0}|\Pi^\prime)= \lbrace 2,0,1,1 \rbrace$; $r(v_{i,0} |\Pi^\prime)= \lbrace 0,3-i,i,1 \rbrace$, for $i=2,3$; $r(v_{1,j}|\Pi^\prime)= \lbrace 2,0,3-j,j \rbrace $, for $j=1,2$; $r(v_{i,j}|\Pi^\prime)= \lbrace 3-i,i,j-3,4-j \rbrace$, for $i=1,2$ and $j=3,4$; $r(v_{i,1} |\Pi^\prime)= \lbrace 0,2,i,1 \rbrace $, for $i=2,3$; $r(v_{i,2} |\Pi^\prime)= \lbrace 0,4-i,2i-3,2 \rbrace $, for $i=2,3$; and $r(v_{3,j}|\Pi^\prime)= \lbrace 1,j-3,3,4-j \rbrace$, for $j=3,4$.

We can see that for every two vertices $x,y\in{V(K_3 \odot W_4 )}$, we have $r(x|\Pi^\prime)\neq r(y|\Pi^\prime)$. Therefore $\Pi^\prime$ is a resolving partition of graph $K_3 \odot W_4$. Consequently, the upper bound of the partition dimension of the graph $K_3 \odot W_4$ is $pd(K_3 \odot W_4 )\leq 4$. Based on the lower bound and upper bound of of the graph $K_3 \odot W_4$, we get $pd(K_3 \odot W_4 )=4$.\\
	Else, based on Lemma 3.3, it is obtained that $pd(K_n \odot W_{n}) > n-1$.And based on Lemma 3.4, it is obtained that $pd(K_n \odot W_{n}) \leq  n$.. thus it can be concluded that $pd(K_n \odot W_{n}) \leq  n$ for $n \geq 4$.
\end{proof}
\begin{theorem}
	For Every $n, n \geq 3$
\begin{align*}
	pd(K_n \odot W_{n+1}) =\begin{cases}
		4, n=3 \\
		n, n \geq 4
	\end{cases}
\end{align*} 
\end{theorem}
\begin{proof}
	For $n=3$,
	Let graph $K_3 \odot W_5$ with $V(K_3 \odot W_5)=\{V(k_3) \cup {\cup_{i \in V(K_3)V(W^i_3)}}\}$ and $E(K_3 \odot W_5)=\{E(k_3) \cup \bigcup_{i \in E(K_3)V(W^i_3)} \cup \{u_i,v_{i,j}| 1 \leq i \leq 3, 0 \leq j \leq 5\}$ will be shown $pd(K_3 \odot W_{5}) >3$\\
	
	Let a partition $\prod$ of the graph $K_3 \odot W_5$ be constructed such that $\prod=\{S_1, S_2, S_3\}$ with all vertices $K_3 \odot W_5$ partitioned arbitrary into $\prod$. Take $V(W^{i'}_5)$ for any $i \in \{1,2,3\}$. Based on Proposition 4.1.1, there are at least five weakly equivalent vertices. Since there are only three partition classes in $\prod$, there must be two weakly equivalent vertices in the same partition class, say $v{i,k}$ and $v_{i,j}$ where $j,k \in \{1,2,3,4,5\}$ are in $S_j$. Consider the following cases:\\
	
	\textbf{Case I} $v{i,k}$ and $v_{i,j}$ are adjacent vertices, say  $v{i,k}$ and $v_{i,k+1(\operatorname{mod} 5)}$
	\begin{itemize}
		\item [$\circ$]  If the neighbors of $v{i,k}$ and $v_{i,k+1(\operatorname{mod} 5)}$ i.e  $v_{i,k-1(\operatorname{mod} 5)}$ and  $v_{i,k+2(\operatorname{mod} 5)}$ 	are in the same partition class, say $S_m$, then according to Lemma 3.2, $\prod$ is not a resolving set. 
		\item [$\circ$]  If the neighbors of $v{i,k}$ and $v_{i,k+1(\operatorname{mod} 5)}$ are in different partition class, say $S_m$ and $S_n$, then $r(v_{i,k}|\prod' )=r(v_{i,k-1(\operatorname{mod} 5)}|\prod ')=(0,2,1)$, so $\prod$ is not a resolving set. Let $v_{i,k-1(\operatorname{mod} 5)}$ is in $S_l$. Considering $u_i$ and $v_{i,0}$ are two strongly equivalent vertices on the block, the partition set $\prod$ is not a resolving partition if $u_i$ and $v_{i,0}$ are in the same partition class. Suppose $u_i$ and $v_{i,0}$ are in a different partition, say $u_i$ in $S_m$. If $v_{i,0}$ in $S_n$, $r(v_{i,k}|\prod')=r(u_i|\prod ')=(0,1,1)$ and $S_l$, $r(v_{i,k+2}|\prod')=r(u_i|\prod ')=(0,1,1)$
	\end{itemize}

\textbf{Case II} $v{i,k}$ and $v_{i,j}$ are non-adjacent vertices, say  $v{i,k}$ and $v_{i,k+1(\operatorname{mod} 5)}$\\
Consider the remaining three weakly equivalent vertices to be partitioned into $\prod$, say $v_{i,k-2},v_{i,k-1}$ and $v_{i,k+1}$.If any of these three vertices are partitioned into $s_l$ then there will be two weakly equivalent vertices partitioned into the same partition class, returning to Case 1, making $\prod$ not a resolving set. The remaining three vertices must be partitioned into $S_m$ and $S_n$ with no two adjacent vertices partitioned into the same class. Suppose $v_{i,k+1}$ and $v_{i,k-2}$  in $S_n$ and $,v_{i,k-1}$  in $S_m$. Considering $v_{i,0}$ is adjacent to all five weakly equivalent vertices, if $v_{i,0}$  is partitioned into a class, then the distance to the other partitions is always 1, or $d(v_{i,0}, S_m)=d(v_{i,0}, S_n)=1$.In this case, there will always be a vertex in each partition class with a distance of 1 to the other partition classes, so there will always be $x \in \{v_{i,k-2},v_{i,k-1}, v_{i,k}, v_{i,k+1},v_{i,k+2}\}$ such that $r(x|\prod')=r(v_{i,0}|\prod ')=(0,1,1)$. Thus, it can be said that any partition set consisting of only three partition classes always has two vertices with the same representation, making $\prod$ where $|\prod|=3$ not a resolving partition set. Therefore, $pd(K_3 \odot W_5)>3$.\\
We defined partition $\prod=\{S_1, S_2, S_3, S_4\}$ bellow\\
	$ S_1=\{ v_{1,0}, v_{1,1}, v_{1,2}, v_{1,5}, v_{2,0}, v_{2,1}, v_{2,2}, v_{2,5}, u_1, u_2\}$ \\
$ S_2=\{v_{1,4}, v_{1,5}, v_{2,4}, v_{3,4}\}$ \\
$ S_3=\{v_{3,0}, v_{3,1}, v_{3,2}, v_{3,5}, u_3\}$ \\
$ S_2=\{v_{3,3}, v_{2,3}\}$ \\
Easily we can check for the representation for every vertex of graph $K_3 \odot W_5$  respect to $\prod'$ have a different representation. Thus $\prod'$ is a resolving set and also $pd(K_3 \odot W_5)\leq 4$. Consequently $pd(K_3 \odot W_5)=3$.

Based on Lemma 3.3, it is obtained that $pd(K_n \odot W_{n+1})> n-1$. And based on Lemma 3.6, it is obtained that $pd(K_n \odot W_{n+1})\leq >n-1$. thus it can be concluded that $pd(K_n \odot W_{n+1})=n$.for $n\geq 4$. 
\end{proof}


\end{document}